\documentclass[final,leqno,onethmnum,onetabnum]{siamltex1213}
 \usepackage{dsfont}   
 \usepackage{amsfonts}

\newcommand{\nn}{\nonumber}

\newcommand{\beq}{\begin{equation}}
\newcommand{\eeq}{\end{equation}}
\newcommand{\Nb}{{\mathbb N}}
\newcommand{\T}{\tau}
\newcommand{\trmi}{t \rightarrow +\infty}
\newcommand{\ve}{\varepsilon}
\renewcommand{\t}{\tau}

\newtheorem{remark}{Remark}

\title{Rates of convergence to scaling profiles in a submonolayer 
deposition model and the preservation of memory of the initial condition}

\author{Fernando P. da Costa\footnotemark[2]\ \footnotemark[4]
\and Jo\~ao T. Pinto\footnotemark[3]\ \footnotemark[4]
\and Rafael Sasportes\footnotemark[2]\ \footnotemark[4]}

\begin{document}

\maketitle

\renewcommand{\thefootnote}{\fnsymbol{footnote}}

\footnotetext[2]{Departamento de Ci\^encias e Tecnologia, Universidade Aberta, Lisboa, Portugal, and
Centro de An\'alise Matem\'atica, Geometria e Sistemas Din\^amicos, Instituto Superior T\'ecnico,
Universidade de Lisboa, Lisboa, Portugal}
\footnotetext[3]{Departamento de Matem\'atica, and Centro de An\'alise Matem\'atica, Geometria e Sistemas Din\^amicos,
Instituto Superior T\'ecnico, Universidade de Lisboa, Lisboa, Portugal}
\footnotetext[4]{Partially funded by FCT/Portugal through project RD0447/CAMGSD/2015.}

\pagestyle{myheadings}
\thispagestyle{plain}
\markboth{F.~P. DA COSTA, J.~T. PINTO AND R. SASPORTES}{RATES OF CONVERGENCE TO SCALING PROFILES}



\begin{abstract}
We establish rates of convergence of solutions to scaling (or similarity) profiles in a coagulation type system modelling 
submonolayer deposition. We prove that, although all memory of the initial condition is lost in the similarity limit, information
about the large cluster tail of the initial condition  is preserved in the rate of approach to the similarity profile.
The proof relies in a change of variables that allows for the decoupling of the original infinite system of ordinary differential equations into a closed 
two-dimensional nonlinear system for the monomer--bulk dynamics and a lower triangular infinite dimensional linear one for the cluster
dynamics. The detailed knowledge of the long time  
monomer concentration, which was obtained earlier by Costin et al.
in \cite{costin} using asymptotic methods and is rederived here by center manifold arguments, is then used for the asymptotic
evaluation of an integral representation formula for the concentration of $j$-clusters. The use of higher order expressions, both
for the Stirling expansion and for the monomer evolution at large times allow us to obtain, not only the similarity limit, but also
the rate at which it is approached.
\end{abstract}




\section{Introduction}\label{sec:intr}

Submonolayer deposition is
the process of particle deposition onto a surface such that
the deposited particles can diffuse and coagulate to form clusters, and those clusters are so sparse 
as \emph{not} to cover the full original surface.  The theoretical modelling of this
process can be done using a variety of approaches \cite{Mul09}. Our work is a continuation of a
number of recent mathematical studies of the dynamics of deposition using a mean-field approach \cite{cps,crw,cs,costin}.

Denoting by $c_j=c_j(t)$ the concentration number of clusters of size $j$, or $j$-clusters, at time $t$, on the
deposition surface, assuming this surface to be bombarded at a constant rate $\alpha>0$ by 
1-clusters (also called \emph{monomers}), and considering the
clusters' concentrations so small that cluster-cluster reactions can be disregarded, the cluster dynamics on 
the surface can be modelled by a Smoluchowski's coagulation system with Becker-D\"oring like coagulation kernel, meaning that
the only allowed coagulation reactions are those in which a monomer takes part, namely $(1)+(j)\rightarrow (j+1)$,
 with rate coefficients $a_{1,j}\geq 0$, and $j\in\{1,2,\ldots\}$. 

An assumption that is relevant in some applications is the existence of a critical cluster size $n$ below which
clusters are not stable and do not occur in the system in any significant amount in the time scale
of the coagulation reactions. There are several ways to model this assumption (see, e.g., \cite{costin} and
references therein). In \cite{costin} it was modelled by considering that no
clusters of size larger than $1$ and smaller than or equal to $n-1$ can exist, and thus the smaller
cluster that is not a monomer has size $n$ and is formed when $n$ monomers come together and react
into an $n$-cluster (if, as in Monte Carlo
simulations, we consider the monomers sitting in the vertices of a lattice, we could have these ``multiple''
collisions simply by having $n-1$ monomers surrounding an empty site ---as in the four nearest
neighbours in a square lattice--- which is suddenly bombarded by an aditional monomer, creating the $n$-cluster
in a single $n$ body reaction).

If one considers that the
coagulation rates are independent of the cluster sizes undergoing reaction ($a_{1,j}\equiv 1$, say), the mean-field model 
for submonolayers deposition with a critical cluster size $n$ becomes the following 
infinite system of ordinary differential equations:

\beq
\left\{
\begin{array}{lcl}
\dot{c}_1 & = & \alpha - nc_1^n - c_1 \displaystyle{\sum_{j=n}^{\infty}c_j} \\
\dot{c}_n & = & c_1^n-c_1c_n,  \\
\dot{c}_j & = & c_1c_{j-1}-c_1c_j, \;\;\; j \geq n+1. 
\end{array}
\right.\label{system}
\eeq

System (\ref{system}) with $n=2$ (i.e., with no unstable clusters) 
was considered in \cite{cps,crw,cs}, where the long time behaviour
of solutions and the approach to a similarity profile was studied; in \cite{cps,cs} 
it was even considered the case of
time dependent monomer input $\alpha=\alpha(t)$ of power law type.

In \cite{costin} system (\ref{system}) was considered with general $n\geq 2,$ and the
results in \cite{crw} were correspondingly generalized to the following

\begin{theorem}[\mbox{\cite[Theorem 5]{costin}}]\label{inputmonomers}
Let $(c_j)$ be any solution to (\ref{system}). Consider the new time scale $\t(t):=\int_{t_0}^tc_1(s)ds,$ 
and let $\widetilde{c}_j(\t):=c_j(t(\t)).$  Then, 
\[
\displaystyle{
\lim_{{\begin{array}{c}j,\,\t \rightarrow +\infty\\ \eta=j/\t\; \mbox{\rm  fixed}\\ \eta \neq 1\end{array}}}\!\!
\left(\frac{n\t}{\alpha}\right)^{(n-1)/n}\!\!
\widetilde{c}_j(\t) = \Phi_1(\eta) :=
\left\{\begin{array}{ll} \left(1-\eta\right)^{-(n-1)/n},& \mbox{\rm  if $\;0<\eta<1$}\\
0, & \mbox{\rm  if $\;\eta>1$}.
\end{array}\right.
}
\]
\end{theorem}

This result means that, in the self-similar (or similarity) variable $\eta=j/\t$, the scaled solutions approach a universal profile
$\Phi_1(\eta)$ in the similarity limit $j, \t\to\infty$ with $\eta$ constant.

The purpose of the present paper is to study the rate at which the limit function
$\Phi_1(\eta)$ in Theorem~\ref{inputmonomers} is approached. This is a kind of problem that is 
rather natural in long time dynamics, with the paradigmatic example being the linearization method
which, when applicable, provides an exponential estimate for the rate of convergence of solutions to 
the limit (an equilibrium, say).

In the context of Smoluchowski's coagulation equations and analogous systems (such as coa\-gu\-la\-tion-fragmentation, and
Bercker-D\"oring systems) the proofs of existence of, and convergence to, self-similar solutions 
are already so demanding that the study of the rate at which this convergence takes place is, so far, typically out of reach.

To the best of our knowledge, so far only the two papers \cite{cmm2010,sri}
considered the problem of the rate of converge to self-similar behaviour in Smoluchowski's systems.
In \cite{cmm2010} the result was obtained for the constant kernel case (in continuous variables) $a_{x,y}\equiv 1$  using the approach of
linearizing the coagulation equation in self-similar variables about the self-similar solution, 
and proving that the linearized operator has a spectral gap in an appropriate
scale of weighted Sobolev spaces. In \cite{sri} the case of so called solvable kernels (i.e. $a_{x,y}=1,\, x+y,\,xy$) was considered
using an approach analogue to the Berry-Ess\'een theorem in classical probability theory. Both approaches require rather delicate
and difficult analysis.

In the present paper we are able to approach this rate of convergence problem using an essentially simpler approach due
to the special nature of the coagulation reactions in (\ref{system}). Exploring the approach which was already
used to good effect in the proof of Theorem~\ref{inputmonomers} in \cite{crw,costin}, we can establish the main
theorem of this paper:

\begin{theorem}\label{teo:main} 
Let $(c_j)$ be any non-negative solution of (\ref{system}) with initial data satisfying
$\exists \rho_1, \rho_2>0, \mu > 1: \forall j\geq n, j^{\mu}c_j(0)\in [\rho_1, \rho_2]$.
Assume the notation introduced in Theorem~\ref{inputmonomers}.
Then, as $j,\,\tau \rightarrow +\infty,$ with $\eta=j/\tau\neq 1$ fixed,
\begin{eqnarray}
 \lefteqn{
\left|\left(\frac{n\T}{\alpha}\right)^\frac{n-1}{n}
\widetilde{c}_j(\tau) - \Phi_1(\eta)\right|\;\; \sim} \qquad\mbox{}  \label{convratelhs}  \\
& \sim & \;\;\;(n-1)\left(1-\frac{1}{n}\right)(1-\eta)^{-(n-1)/n}
\frac{\log((1-\eta)\tau)}{(1-\eta)\tau}\mathds{1}_{(0,1)}(\eta) + \label{convrate1} \\ 
 & & + \left(\frac{n}{\alpha}\right)^{(n-1)/n}\!\!\mathcal{O}(1)\;
 \eta^{-\mu}\tau^{\frac{n-1}{n}-\mu}\mathds{1}_{(1,\infty)}(\eta).   \label{convrate2} 
\end{eqnarray}
\end{theorem}

What we find interesting in this result is that, although Theorem~\ref{inputmonomers} states that,
in the similarity variable $\eta,$ solutions to (\ref{system}) approach a universal profile independent of the initial condition
(indeed this universal behaviour is what is physically relevant in this type of enquiry), the
rate at which this limit behaviour is approached still preserves some information about the initial
condition, namely the rate of decay of the initial datum for large cluster sizes, $\mu$, can still be
computed from the observation of the decay rate of scaled solutions to $\Phi_1(\eta),$ for $\eta >1.$


\section{Preliminaries}\label{sec:prelim}

Our approach to the study of system (\ref{system}) follows the one used in \cite{crw, costin} and consists of 
the exploration of the following two observations:

\begin{romannum}
\item First note that the equation for $c_1$ depends only on $c_1$ and on the ``bulk'' quantity $y(t):=\sum_{j=n}^{\infty}c_j(t)$, which 
(formally) satisfies the differential equation $\dot{y}=c_1^n$. Thus, the definition of this bulk variable allow
us to decouple the resulting infinite dimensional system into a closed two-dimensional system for the 
monomer--bulk variables $(c_1, y)$,
from which we get all the needed information about the behaviour of $c_1$.
\item Secondly, the remaining equations for \(c_j\), with \(j\geq n\), depend only on those same variables $c_j$, and on $c_1$.
However, the way they depend on $c_1$ is such that, by an appropriate change of the time variable, the system is transformed
into a linear lower triangular infinite system of ordinary diferential equations, which can be recursively solved,
in terms of \(c_1\), using the variation
of constants formula.
\end{romannum}

Let us look more closely at the details of  (ii) first:
writing the second and third equations of (\ref{system}) in the form
\begin{equation}
  \left\{
  \begin{array}{l}
  {\dot{c}_{n}}=c_1({c_1}^{n-1}-c_n)\\
  \dot{c}_{j}=c_1(c_{j-1}-c_j), \;\;\; j \geq n+1,
  \end{array}
  \right.\label{eq:bixc}
\end{equation}
it is natural to introduce a new time scale
\begin{equation}\label{eq:tau}
\tau(t):=\int_0^t\!c_1(s)\,ds,
\end{equation}
along with scaled variables
\begin{equation}\label{eq:ctilde}
\widetilde{c_1}(\tau):=c_1(t(\tau)) \;\;\; \mbox{\rm   and }\;\;\;  \widetilde{c}_j(\tau):=c_j(t(\tau)), \;j\geq n+1,
\end{equation}
where $t(\tau)$ is the inverse function of $\tau(t)$.
With this new time  (\ref{eq:bixc}) reads 
\begin{equation}
  \left\{
  \begin{array}{l}
   {\widetilde{c}_n}{'}=\widetilde{c_1}^{n-1}-\widetilde{c}_n\\
 {\widetilde{c}_j}{'}=\widetilde{c}_{j-1}-\widetilde{c}_j,\;\;\; j \geq n+1,
  \end{array}
  \right.\label{eq:bixctil}
 \end{equation}
where \((\cdot)'=d/d\tau\).
The equation for \( {\widetilde{c}_n}\) can readily be solved in terms of \(\widetilde{c_1}\), 
\[
{\widetilde{c}_n}(\T)=e^{-\T}{\widetilde{c}_n}(0)+\int_0^\T e^{-(\T-s)}\widetilde{c_1}^{n-1}(s)ds=e^{-\T}{\widetilde{c}_n}(0)+\int_0^\T e^{-s}\widetilde{c_1}^{n-1}(\T-s)ds,
\]
and then we can use this solution to solve the \(c_{n+1}\) equation, and afterwards solve recursively for \(j\geq n+2\). 
The following expression for \(\widetilde{c}_j(\T)\) is obtained in \cite{costin},  
generalizing an analogous result first proved in \cite{crw} for the special case \(n=2\):

\begin{eqnarray}
\widetilde{c}_j(\T) &=& e^{-\T}\sum_{k=n}^{j}\frac{\T^{j-k}}{(j-k)!}c_k(0) +
\frac{1}{(j-n)!}\int_0^{\T}\left(\widetilde{c}_1(\T-s)\right)^{n-1}s^{j-n}e^{-s}ds,\;\;\; \forall j\geq n. \nn \\  
& & \label{cjtausolution}
\end{eqnarray}

Now, the information about $\widetilde{c}_1$ needed to study (\ref{cjtausolution}) has to be extracted from the two-dimensional system for $(c_1,y)$
refered to in observation (i). Let us see this with a bit more detail. Denoting \( x(t):= c_{1}(t)\) and \( y(t):=\sum_{j=n}^{\infty}c_{j}(t) \)
one can easily conclude that these quantities (formally) satisfy the two-dimensional system
\begin{equation}
  \left\{
  \begin{array}{l}
 {\dot{x}}=\alpha -nx^n-xy\\
 {\dot{y}}=x^n.
  \end{array}
  \right.\label{eq:xy}
\end{equation}

It can be proved \cite[Theorem 1]{costin} that if \(\sum_{j=n}^\infty c_j(0) <\infty\), then any solution of system 
(\ref{eq:bixctil})--(\ref{eq:xy})
is also a solution of system (\ref{system}).

Clearly, the monomer--bulk system (\ref{eq:xy}) form a closed two-dimensional system the solutions of which can be studied
independently of what happens to the infinite system (\ref{eq:bixctil}). In particular, the needed information about 
$\widetilde{c}_1$ in (\ref{cjtausolution}) will be obtained from (\ref{eq:xy}).


\section{Center manifold analysis of the monomer--bulk system}\label{sec:cma}

In this section we will study the long time behaviour of solutions to (\ref{eq:xy}).
The following result, proved in \cite{costin}, holds:

\begin{proposition}{\cite[Proposition 1]{costin}}\label{prop1costin}
Let \( (x(\cdot),y(\cdot)) \) be any solution of (\ref{eq:xy}) with non-negative initial condition. Then:
 \begin{romannum}
  \item \( (x(t),y(t)) \) is positive when $t>0.$
  \item \( (x(t),y(t)) \) exists for all \(t>0\).
  \item \(x(t)\to 0, \mbox{   as } t\to+\infty\).
  \item \(y(t)\to +\infty, \mbox{   as } t\to+\infty\).
  \item \(\alpha -xy\to 0, \mbox{  as }  t\to+\infty\).
 \end{romannum}
\end{proposition}

The next result establishes the rate of decay of \(x(t)\) as $t\to\infty$.

\begin{theorem}{\cite[Corrected Eq. (25)]{costin}}\label{teo:ctlimit}
Let \( (x(\cdot),y(\cdot)) \) be any non-negative solution of (\ref{eq:xy}). Then, 
 \beq
 x(t) = \left(\frac{\alpha}{n+1}\right)^{\frac{1}{n+1}}t^{-\frac{1}{n+1}} + \frac{n(n-1)}{n+1}t^{-1}
 + o(t^{-1}), \quad\mbox{ as $t\to +\infty$}. \label{ctlimit}
 \eeq
\end{theorem}

\begin{remark}
 In \cite{costin} the coefficient of the $t^{-1}$ term is slightly different from that
 in (\ref{ctlimit}), in what we believe to be a minor mistake.
\end{remark}

\begin{proof} (of Theorem~\ref{teo:ctlimit}).
The proof of this result in \cite{costin} was done using asymptotic methods. Here, we
follow the idea of the proof of the case $n=2$ in \cite[Proposition 3.2]{crw}, which uses
center manifold methods to obtain the first term in the right-hand side of (\ref{ctlimit}). 
The main difference relative to \cite{crw} being that now we need to use a higher order approximation
to the center manifold in order to obtain the expansion with an additional term, (\ref{ctlimit}).
However, somewhat surprisingly, this has to be done in a two step process, as explained further down.

Let us define a new variable \(v(t):=\alpha -x(t)y(t)\), so that \((x(t),v(t))\to (0,0)\) as \(t\to+\infty\).

By changing variables $(x, y)\mapsto (x, v)$ system (\ref{eq:xy}) becomes
\begin{equation}
\left\{\begin{array}{l}
 \dot{x}=v-nx^n\\
  \displaystyle{\dot{v}}=-\frac{\alpha v}{x}-x^{n+1}+\frac{v^2}{x}+\alpha nx^{n-1}-nvx^{n-1}.
  \end{array}
  \right.\label{eq:vx}
\end{equation}
Suppose $x(0)\neq 0$. Otherwise, by Proposition~\ref{prop1costin} (i) and (ii)
we know that $x(t)>0$ for all $t>0,$ and so just redefine time so
that $x(0)$ becomes positive.
Now change the time scale
\begin{equation}
t\mapsto \zeta = \zeta(t):= \int_0^t\frac{1}{x(s)}ds,\label{tzeta}
\end{equation}
and define $\left(\widetilde{x}(\zeta),\widetilde{v}(\zeta))\right):= \left(x(t(\zeta)), v(t(\zeta))\right),$
where $t(\zeta)$ is the inverse function of $\zeta(t)$. By Proposition \ref{prop1costin}(iii),
we have $x(t)\rightarrow 0$ as $\trmi$, and so also $\zeta\rightarrow +\infty$ as $\trmi.$ 
With the new time scale system (\ref{eq:vx}) becomes
\begin{equation}
\left\{\begin{array}{l}
 \widetilde{x}'=\widetilde{v}\widetilde{x}-n\widetilde{x}^{n+1}\\
 \widetilde{v}'=-\alpha\widetilde{v}-\widetilde{x}^{n+2}+\widetilde{v}^2+\alpha n \widetilde{x}^n-n \widetilde{v}\widetilde{x}^n,
  \end{array}
  \right.\label{eq:vxtil}
\end{equation}
where  $(\cdot)' = d/d\zeta$, or in matrix form
\begin{equation}
\left[
 \begin{array}{c}
  \widetilde{x}\\ \widetilde{v}
 \end{array} \right]'
 =
 \left[
 \begin{array}{cc}
  0&0\\
  0&-\alpha
 \end{array} \right]
 \left[
 \begin{array}{c}
  \widetilde{x}\\\widetilde{v}
 \end{array} \right]
+
\left[
\begin{array}{c}
  \widetilde{v}\widetilde{x}-n\widetilde{x}^{n+1}\\
  -\widetilde{x}^{n+2}+\widetilde{v}^2+\alpha n \widetilde{x}^n-n \widetilde{v}\widetilde{x}^n
 \end{array}.
 \right]\label{tildevx}
\end{equation}

Since, by Proposition~\ref{prop1costin} (iii)--(v), we know that when $t\to +\infty$ 
all non-negative solutions $(x,y)$ of system
(\ref{eq:xy}) converge to $(0,\infty)$ and $v\to 0$, the transformation above imply that the corresponding
solutions $(\widetilde{x}, \widetilde{v})$ of (\ref{tildevx}) converge to $(0,0)$ when $\zeta\to +\infty.$

Using standard results in center manifold theory \cite[chap.\ 2]{c81} it is straightforward to conclude that 
(\ref{tildevx}) has a center manifold in a neighbourhood of the origin,
that locally exponentially  attracts all orbits, and is the graph of
a function $\widetilde{v}= \phi_n(\widetilde{x})$, where 
\begin{equation}
\phi_n(\widetilde{x}) = n\widetilde{x}^{\,n} -
\frac{1}{\alpha}\widetilde{x}^{\,n+2} +\frac{n(n-1)}{\alpha^{2}}\widetilde{x}^{\,2n+2}  - \frac{n+1}{\alpha^3}\widetilde{x}^{\,2n+4}
+ {\mathcal O}(\widetilde{x}^{\,3n+2}).\label{cmfunction}
\end{equation}

In order to get an expression for the behaviour of solutions $\widetilde{x}(\zeta)$ for large values of $\zeta$, from which (\ref{ctlimit})
can be deduced, we need to proceed in two steps: first, using the first two terms in the expression of the center manifold $\phi_n$
and the knowledge that $\widetilde{x}(\zeta)\to 0$ as $\zeta\to +\infty$, we obtain a first, ``lower order'', expression for the 
long-time behaviour of $\widetilde{x}(\zeta)$. Then, using the next term in the right-hand side of (\ref{cmfunction}) and this ``lower order'' information,
we obtain a ``higher order'' correction to the long-time behaviour of $\widetilde{x}(\zeta)$ which will be sufficient to prove (\ref{ctlimit}).

\begin{remark}
It is worth noting that the two stage process just described cannot be abbreviated to a single step, taking from the start the first
three terms in $\phi_n$ and the knowledge that $\widetilde{x}\to 0$ as $\zeta\to +\infty$. Doing this one would arrive at   
(\ref{ctlimit}) with the equality substituted by an asymptotic equality, $\sim$, meaning that the long-time
limit of the ratio of the left-hand side of (\ref{ctlimit})
by its right-hand side is equal to $1$. Note that this is a weaker result than the one stated in (\ref{ctlimit}), and, in fact, 
it is not strong enough to prove our main result, Theorem~\ref{teo:main}.
\end{remark}

Let us start to implement the idea described above. Considering the approximation to the center manifold given by 
taking only the fist two terms in the right-hand side of (\ref{cmfunction}),
\[
 \phi_n(\widetilde{x}) = n\widetilde{x}^{n} -
\frac{1}{\alpha}\widetilde{x}^{n+2} +{\mathcal O}(\widetilde{x}^{2n+2}),
\]
the dynamics on the center manifold is given by
\[
\widetilde{x}_c' = - \frac{1}{\alpha}\widetilde{x}_c^{n+3} +
{\mathcal O}(\widetilde{x}_c^{2n+3})\quad\mbox{\rm as}\quad\zeta\rightarrow +\infty,
\]
which can be written as
\begin{equation}
 \frac{d}{d\zeta}\left(\frac{\alpha}{(n+2)\widetilde{x}_{c}^{n+2}}\right)= 1+\mathcal{O}(\widetilde{x}_c^n),\quad 
\mbox{\rm  as }\zeta\rightarrow +\infty.\label{diff_1storder_1}
\end{equation}
Knowing that $\widetilde{x}(\zeta)\to 0$ as $\zeta\to + \infty$, the right-hand side of this equation converges
to $1$ and we conclude that, for all $\ve>0,$ there exists a $T>\zeta_0$ such that, for all $\zeta>T,$ 
the following  inequalities hold
$$
1-\varepsilon \leq\frac{d}{d\zeta}\left(\frac{\alpha}{(n+2)\widetilde{x}_{c}^{n+2}}\right) \leq1+\varepsilon.
$$ 
Integrating these differential inequalities between $T$ and $\zeta$ we get
\begin{equation}
(1-\ve)(\zeta -T) \leq\frac{\alpha}{(n+2)\widetilde{x}_{c}^{n+2}} - \frac{\alpha}{(n+2)\widetilde{x}_{T}^{n+2}} \leq 
(1+\ve)(\zeta -T),\label{diff_1storder_2} 
\end{equation}
where $\widetilde{x}_T =\widetilde{x}_c(T).$ 
Dividing (\ref{diff_1storder_2}) by $\zeta$ and taking  $\displaystyle{\liminf_{\zeta\to+\infty}}$
and $\displaystyle{\limsup_{\zeta\to+\infty}}$ we obtain
\[
 1-\ve \leq\liminf_{\zeta\to+\infty}\frac{\alpha}{(n+2)\,\zeta\,\widetilde{x}_{c}^{n+2}} \leq
 \limsup_{\zeta\to+\infty}\frac{\alpha}{(n+2)\,\zeta\,\widetilde{x}_{c}^{n+2}} \leq1+\ve,
\]
which, due to the arbitrariness of $\ve$, means that 
\(
 \lim_{\zeta\to +\infty}\textstyle{\frac{n+2}{\alpha}}\,\zeta\,\widetilde{x}_{c}(\zeta)^{n+2} = 1,
\)
and thus
\begin{equation}
 \widetilde{x}_c^{n+2}(\zeta) = \frac{\alpha}{n+2}\frac{1}{\zeta}(1+ o(1))\quad\mbox{\rm as}\quad\zeta\to + \infty.\label{1storder}
\end{equation}

Let us now consider a better approximation to the center manifold, obtained by considering a further term in the right-hand side of (\ref{cmfunction}),
namely, 
\[
 \phi_n(\widetilde{x}) = n\widetilde{x}^{n} -
\frac{1}{\alpha}\widetilde{x}^{n+2} +\frac{n(n-1)}{\alpha^{2}}\widetilde{x}^{2n+2} +{\mathcal O}(\widetilde{x}^{2n+4}).
\]
The dynamics on the center manifold is now given by 
\beq
\widetilde{x}_c' =-
\frac{1}{\alpha}\widetilde{x}_c^{n+3} +\frac{n(n-1)}{\alpha^{2}}\widetilde{x}_c^{2n+3} 
+{\mathcal O}(\widetilde{x}_c^{2n+5})\quad\mbox{\rm as}\quad\zeta\rightarrow
+\infty.\label{diff}
\eeq
Writing this differential equation as
\[
 \frac{\alpha^2 \widetilde{x}_c'}{-\alpha\widetilde{x}_c^{n+3} +n(n-1)\widetilde{x}_c^{2n+3}}
 = 1+\mathcal{O}(\widetilde{x}_c^{n+2}),\quad 
\mbox{\rm  as }\zeta\rightarrow +\infty,
\]
and observing that
\[
 \int \frac{\alpha^2}{-\alpha{s}^{n+3} +n(n-1)s^{2n+3}} ds=
\frac{\alpha}{(n+2)s^{n+2}}+\frac{n(n-1)}{2s^2} + \psi_n(s),
\]
where 
\[ \psi_{n}(s):=n(n-1)\int \frac{s^{n-3}}{s^n-\frac{\alpha}{n(n-1)}}ds,\]
we conclude that (\ref{diff}) can be written as 
\begin{equation}
 \frac{d}{d\zeta}\left(\frac{\alpha}{(n+2)\widetilde{x}_{c}^{n+2}}+\frac{n(n-1)}{2\widetilde{x}_{c}^{2}}+
 \psi_{n}(\widetilde{x}_{c})\right)= 1+\widetilde{x}_c^{n+2}\mathcal{O}(1),\;
\mbox{\rm  as }\zeta\rightarrow +\infty.\label{5'}
\end{equation}

As in the previous approximation (\ref{diff_1storder_1}), we now need to estimate the right-hand side of
(\ref{5'}), but instead of using only the information that $\widetilde{x}_c(\zeta)\rightarrow 0$ as $\zeta\to+\infty,$ 
we shall use (\ref{1storder}).

Observing that there exist constants $K^* \geq K_*$ such that the right-hand side of (\ref{5'}) can be bounded by
$1+K_*\widetilde{x}_c^{n+2} \leq1+ \widetilde{x}_c^{n+2}\mathcal{O}(1)\leq1+K^*\widetilde{x}_c^{n+2},$
we can write, for all $\zeta > T$,
\begin{equation}
 1+K_*\widetilde{x}_c^{n+2} \leq\frac{d}{d\zeta}\left(\frac{\alpha}{(n+2)\widetilde{x}_{c}^{n+2}}+\frac{n(n-1)}{2\widetilde{x}_{c}^{2}}+
 \psi_{n}(\widetilde{x}_{c})\right) \leq 1+K^*\widetilde{x}_c^{n+2}.\nn
\end{equation}
Integrating these differential inequalities between $T$ and $\zeta>T$ and denoting $\widetilde{x}_c(T)$ by $\widetilde{x}_T,$ we get
\begin{eqnarray}
\lefteqn{\zeta-T + K_*\int_T^\zeta\widetilde{x}_{c}^{n+2}(s)ds \leq} \nonumber \\ & \leq & 
\left(\textstyle{\frac{\alpha}{(n+2)\widetilde{x}_{c}^{n+2}}}+\textstyle{\frac{n(n-1)}{2\widetilde{x}_{c}^2}}+\psi_{n}(\widetilde{x}_{c})\right) 
  - \left(\textstyle{\frac{\alpha}{(n+2)\widetilde{x}_{T}^{n+2}}}+\textstyle{\frac{n(n-1)}{2\widetilde{x}_{T}^2}}+\psi_{n}(\widetilde{x}_{T})\right),\label{1_5''}
\end{eqnarray}
and
\begin{eqnarray}
\lefteqn{\zeta-T + K^*\int_T^\zeta\widetilde{x}_{c}^{n+2}(s)ds \geq}  \nonumber \\ & \geq & 
\left(\textstyle{\frac{\alpha}{(n+2)\widetilde{x}_{c}^{n+2}}}+\textstyle{\frac{n(n-1)}{2\widetilde{x}_{c}^2}}+\psi_{n}(\widetilde{x}_{c})\right) 
  - \left(\textstyle{\frac{\alpha}{(n+2)\widetilde{x}_{T}^{n+2}}}+\textstyle{\frac{n(n-1)}{2\widetilde{x}_{T}^2}}+\psi_{n}(\widetilde{x}_{T})\right).\label{5''}
\end{eqnarray}
We now use (\ref{1storder}) to estimate the integral of $\widetilde{x}_{c}^{n+2}:$ Let $\ve>0$ be fixed arbitrarily, and,
if necessary, redefine $T$ such that $\frac{n+2}{\alpha}\,\zeta\,\widetilde{x}_c^{n+2}(\zeta) \in [1-\ve, 1+\ve],$ for all $\zeta >T.$
Thus
\[
 \frac{\alpha}{n+2}(1-\ve)(\log\zeta -\log T) \leq\int_T^\zeta\widetilde{x}_{c}^{n+2}(s)ds \leq\frac{\alpha}{n+2}(1+\ve)(\log\zeta -\log T).
\]
On the other hand, to estimate $\psi_n(\widetilde{x}_{c}(\zeta))$ as $\zeta\to +\infty$ first observe that, defining
$y:=\left(\frac{n(n-1)}{\alpha}\right)^{1/n}s,$ 
we can write $\psi_{n}(y) = -\alpha\left(\frac{n(n-1)}{\alpha}\right)^{1+2/n}\int\frac{y^{n-3}}{1-y^n}dy.$ 
For $n\geq3$ the explicit expression of this last integral is known from \cite[Integrals 2.146-3, 2.146-4]{gr} 
and it is easily seen to be bounded when $y\to 0.$  The case $n=2$ is a
little trickier: we can easily compute
\[
    \psi_2(s) = 2\int\frac{s^{-1}}{s^2- \frac{\alpha}{2}}ds =  
    -\frac{4}{\alpha}\log\frac{\left(\frac{2}{\alpha}\right)^{1/2}s}{\sqrt{1-\frac{2}{\alpha}s^2}} =
    -\frac{4}{\alpha}\log s + {\mathcal O}(1),\qquad\mbox{\rm  as } s\to 0,
\]
and to evaluate $\psi_{n}(\widetilde{x}_{c}(\zeta))$ when $\zeta\to +\infty$ we now need to use (\ref{1storder}): 
in the case $n=2$, we have $\widetilde{x}_{c}(\zeta) = \left(\frac{4}{\alpha}\zeta\right)^{-1/4}(1+o(1))$
as $\zeta\to + \infty.$ Hence we conclude that, as $\zeta\to +\infty,$
\begin{eqnarray*}
 \psi_2(\widetilde{x}_{c}(\zeta)) & = & -\frac{4}{\alpha}\log
 \left(\frac{4}{\alpha}\zeta\right)^{-1/4} + {\mathcal O}(1)\\
 & = & \frac{1}{\alpha}\log\zeta + {\mathcal O}(1).
\end{eqnarray*}

Now we have the tools to estimate (\ref{1_5''}) and (\ref{5''}) for large values of $\zeta$.
Dividing (\ref{1_5''}) by $\zeta$, taking  $\displaystyle{\liminf_{\zeta\to+\infty}},$ and
noting that when taking this limit all the (constant) terms containing $\widetilde{x}_{T}$, after being divided by
$\zeta$, converge to zero,  we obtain
\begin{eqnarray}
\lefteqn{\liminf_{\zeta\to+\infty}\left( \frac{\alpha}{(n+2)\zeta\widetilde{x}_{c}^{n+2}}+
\frac{n(n-1)}{2\zeta\widetilde{x}_{c}^2}\right) \geq}\nonumber \\ 
& \geq& 
\liminf_{\zeta\to+\infty}\frac{1}{\zeta}\left(\zeta-T + 
\textstyle{\frac{K_*\alpha}{n+2}}(1-\ve)(\log\zeta -\log T) - \psi_n(\widetilde{x}_{c}(\zeta))\right) \nn \\
& = &  1 + \lim_{\zeta\to+\infty}\left(\textstyle{\frac{K_*\alpha}{n+2}}(1-\ve)\textstyle{\frac{\log\zeta}{\zeta}} - 
\textstyle{\frac{1}{\zeta}}\left(T + \textstyle{\frac{K_*\alpha}{n+2}}(1-\ve)\log T\right) - 
\textstyle{\frac{\psi_n(\widetilde{x}_{c}(\zeta))}{\zeta}}\right)\label{liminfcor} \\
& = &  1. \nn 
\end{eqnarray}

Analogously, dividing (\ref{5''}) by $\zeta$ and taking  $\displaystyle{\limsup_{\zeta\to+\infty}}$ we conclude that
\begin{eqnarray}
\lefteqn{\limsup_{\zeta\to+\infty}\left( \frac{\alpha}{(n+2)\zeta\widetilde{x}_{c}^{n+2}}+
\frac{n(n-1)}{2\zeta\widetilde{x}_{c}^2}\right) \leq}\nonumber \\ 
& \leq&  1 + \lim_{\zeta\to+\infty}\left(\textstyle{\frac{K^*\alpha}{n+2}}(1+\ve)\textstyle{\frac{\log\zeta}{\zeta}} - 
\textstyle{\frac{1}{\zeta}}\left(T + \textstyle{\frac{K^*\alpha}{n+2}}(1+\ve)\log T\right) - 
\textstyle{\frac{\psi_n(\widetilde{x}_{c}(\zeta))}{\zeta}}\right)\label{limsupcor} \\
& = &  1. \nn 
\end{eqnarray}

Thus, (\ref{liminfcor})--(\ref{limsupcor}) imply  the existence of a function $F_n$ such that $F_n(\zeta) = {\mathcal O}(1)$
for large $\zeta$, and
\begin{equation}
 \left( \frac{\alpha}{(n+2)\widetilde{x}_{c}^{n+2}}+\frac{n(n-1)}{2\widetilde{x}_{c}^2}\right)
\frac{1}{\zeta}=1 + \frac{\log\zeta}{\zeta}F_n(\zeta), \quad \mbox{\rm  as }\zeta\to+\infty. \label{xclim}
\end{equation}
Let \( \beta:=1/(n+2),\ A:=n(n-1)/2 \) and write (\ref{xclim}) as
\begin{equation}
 \alpha \beta + A\widetilde{x}_{c}^n = \zeta \widetilde{x}_{c}^{n+2} \left(1+ \frac{\log\zeta}{\zeta}F_n(\zeta) \right), 
 \quad \mbox{\rm  as } \zeta\to+\infty.
 \label{eq:xzeta}
\end{equation}
Since $\textstyle{\frac{\log\zeta}{\zeta}}F_n(\zeta)\to 0$ as $\zeta\to +\infty$ we can use the geometric series expansion to write
\[
 \left(1+ \frac{\log\zeta}{\zeta}F_n(\zeta) \right)^{-1} = 1 - \frac{\log\zeta}{\zeta}F_n(\zeta) +  
 {\mathcal O}\left(\left(\textstyle{\frac{\log \zeta}{\zeta}}\right)^2\right),\quad \mbox{\rm  as } \zeta\to+\infty,
\]
and thus, from (\ref{eq:xzeta}), we conclude that, as $\zeta\to +\infty,$
\begin{equation}
  \widetilde{x}_{c}^{n+2}=\left(\frac{\alpha \beta}{\zeta}\right)\left(1 +\textstyle{\frac{A}{\alpha\beta}}\widetilde{x}_{c}^{n}\right)
   \left(1 - \textstyle{\frac{\log\zeta}{\zeta}}F_n(\zeta) + \mbox{\rm h.o.t.}\right),\label{xn+2corr}
\end{equation}
where ``h.o.t.'' denotes terms with order higher than the orders of those explicitly written down.
Taking into account (\ref{1storder}) to estimate $\widetilde{x}_{c}^{n}$ in the right-hand side of (\ref{xn+2corr}), after 
some manipulations and using the binomial expansion we obtain, as \(\zeta\to+\infty\),
\begin{equation}
 \widetilde{x}_{c} = 
 \left(\frac{\alpha\beta}{\zeta}\right)^{\beta}\left(1+\frac{A}{\alpha}\left(\frac{\alpha\beta}{\zeta}\right)^{n\beta} -
 \beta\textstyle{\frac{\log\zeta}{\zeta}}F_n(\zeta) + \mbox{\rm h.o.t.}\right).\label{eq:xz}
\end{equation}

From standard center manifold theory \cite[chap. 2]{c81}, the long-time behaviour of 
$(\widetilde{x}(\zeta), \widetilde{v}(\zeta))$ is determined
by the behaviour on the center manifold modulo exponentially decaying terms
$\mathcal{O}\left(e^{-\lambda \zeta}\right),$ where $\lambda\in (0, \alpha)$,
in particular we can write
$$\widetilde{x}(\zeta) = \widetilde{x}_c(\zeta) + \mathcal{O}\left(e^{-\lambda \zeta}\right),$$
and using (\ref{eq:xz}) and remembering the definitions of $\beta$ and $A$, we conclude that
\begin{equation}
\widetilde{x}(\zeta) = 
 \left(\frac{\alpha}{(n+2)\zeta}\right)^{\frac{1}{n+2}} +
 \frac{n(n-1)}{2\alpha}\left(\frac{\alpha}{(n+2)\zeta}\right)^{\frac{n+1}{n+2}} + 
 \mbox{\rm h.o.t.}.
 \label{xlim}
\end{equation}

In order to obtain the corresponding estimates in the original time variable $t$
we first need to relate the asymptotics of both time scales.

\begin{lemma}\label{le:zetat}
With the same notations and definitions as before, we have
 \begin{equation}
  t  = \textstyle{\frac{(\alpha\beta)^{\beta}}{1-\beta}}\zeta^{1-\beta}+
 \textstyle{\frac{A}{(\alpha\beta)^{\beta}}}\zeta^{\beta} + o\left(\zeta^{\beta}\right), 
 \;\;\mbox{  as } t\to +\infty,\label{eq:tofzeta} 
 \end{equation}
and
 \begin{equation}
 \zeta =\left(\textstyle{\frac{1-\beta}{(\alpha\beta)^{\beta}}}\right)^{\frac{1}{1-\beta}}t^{\frac{1}{1-\beta}}
 -A\left(\textstyle{\frac{1-\beta}{\alpha\beta}}\right)^{\frac{2\beta}{1-\beta}}t^{\frac{2\beta}{1-\beta}} + 
 o\left(t^{\frac{2\beta}{1-\beta}}\right), \;\;\mbox{  as } \zeta\to +\infty.\label{eq:zetat}
 \end{equation}
\end{lemma}

\begin{proof} From (\ref{xlim}) it follows that, for all $\ve>0$, there exists  $T=T(\ve)$ 
such that, for all $\zeta > T$,
\begin{equation}
\left(\widetilde{x}(\zeta)-\left(\frac{\alpha\beta}{\zeta}\right)^{\beta}\right)
\frac{\alpha}{A}\left(\frac{\alpha\beta}{\zeta}\right)^{-(n+1)\beta} \in [1-\ve, 1+\ve].
\label{bound}
\end{equation}

By the definition of $\zeta$ in (\ref{tzeta}) we can write
\[\zeta(t) -\zeta(t_0) =  \int_{t_0}^t\frac{1}{x(s)} ds,\]
and so $d\zeta/dt= 1/x(t).$ Hence
 $dt/d\zeta= \widetilde{x}(\zeta),$ and upon integration
\begin{equation}
 t(\zeta)-t(\zeta_0) = \int_{\zeta_0}^{\zeta}\widetilde{x}(s)ds, \label{tofzeta}
\end{equation}
where $\zeta_0=\zeta(t_0).$ Using the upper bound in (\ref{bound}) we have that, for $\zeta_0\geq T,$
\begin{eqnarray*}
t(\zeta)-t(\zeta_0) &\leq & (1+\ve)(\alpha\beta)^{(n+1)\beta}\frac{A}{\alpha}
\int_{\zeta_0}^{\zeta}
s^{-(n+1)\beta}ds  +(\alpha\beta)^{\beta}\int_{\zeta_0}^{\zeta} s^{-\beta}ds \\
& = & (1+\ve)(\alpha\beta)^{-\beta}A(\zeta^{\beta}-\zeta_0^{\beta})+\textstyle{\frac{(\alpha\beta)^{\beta}}{1-\beta}}
(\zeta^{1-\beta}-\zeta_0^{1-\beta}),
\end{eqnarray*}
from which we conclude that
\[
 t(\zeta) \leq \textstyle{\frac{(\alpha\beta)^{\beta}}{1-\beta}}\zeta^{1-\beta}+
 (1+\varepsilon)\textstyle{\frac{A}{(\alpha\beta)^{\beta}}}\zeta^{\beta}+
 \mathcal{O}(1) \;\mbox{\rm  as }\, \zeta\to\infty.
\]
By considering the lower bound in (\ref{bound}) we would get the reversed inequality, with \(1+\varepsilon\) replaced by 
\(1-\varepsilon\), and since \(\varepsilon\) is arbitrary, combining these two inequalities we conclude (\ref{eq:tofzeta}).

In order to obtain \(\zeta\) as a function of \(t\) we first consider the first term in the expansion (\ref{eq:tofzeta}),
written in the form
\(t(\zeta) =(\alpha\beta)^{\beta}(1-\beta)^{-1}\zeta^{1-\beta}(1+o(1))\), 
which implies \(\zeta^{1-\beta}=(\alpha\beta)^{-\beta}(1-\beta)\,t\,(1+o(1)),\) 
and we substitute this into the lower order term \(\zeta^\beta\) in (\ref{eq:tofzeta}), obtaining, as \(\zeta\to\infty,\)
\begin{eqnarray*}
 t(\zeta) &=&\frac{(\alpha\beta)^{\beta}}{1-\beta}\zeta^{1-\beta}+ 
A(\alpha\beta)^{-\beta}\zeta^{\beta}+ o(\zeta^\beta)\\
&=& \frac{(\alpha\beta)^{\beta}}{1-\beta}\zeta^{1-\beta}+ 
A(\alpha\beta)^{-\beta}\left(\zeta^{1-\beta}\right)^{\frac{\beta}{1-\beta}}+ o(\zeta^\beta)\\
&=& \frac{(\alpha\beta)^{\beta}}{1-\beta}\zeta^{1-\beta}+ 
A(\alpha\beta)^{-\beta}\left((\alpha\beta)^{-\beta}(1-\beta)t(1+o(1))\right)^{\frac{\beta}{1-\beta}}
+ o\left(t^{\frac{\beta}{1-\beta}}\right)\\
&=& \frac{(\alpha\beta)^{\beta}}{1-\beta}\zeta^{1-\beta}+ 
A\left(\frac{1-\beta}{\alpha\beta}\right)^{\frac{\beta}{1-\beta}}t^{\frac{\beta}{1-\beta}}(1+o(1))
+ o\left(t^{\frac{\beta}{1-\beta}}\right).
\end{eqnarray*}
Write the last expression as
\[
\zeta^{1-\beta} = (1-\beta)(\alpha\beta)^{-\beta}t\left(1-A  
\left(\frac{n+1}{\alpha}\right)^{\frac{1}{n+1}}t^{-\frac{n}{n+1}}+o\left(t^{-\frac{n}{n+1}}\right) \right).
\]
Raising both sides to the power $\frac{1}{1-\beta}$ 
and using Newton's binomial series in the right-hand side we obtain (\ref{eq:zetat}).\end{proof}

To finish the proof of Theorem~\ref{teo:ctlimit} we
use the expression (\ref{eq:zetat}) for \(\zeta(t)\) in (\ref{xlim}). After a few rearrangements similar to those
above we obtain (\ref{ctlimit}).
\end{proof}
%
%
%
%

\section{Monomer long time behaviour in the modified time scale}\label{sec:cts}

As pointed out in Section~\ref{sec:prelim}, the proof of Theorem~\ref{teo:main} relies in the exploration of (\ref{cjtausolution}),
for which one needs detailed information on the long-time behaviour of $\widetilde{c}_1(\T).$ This will be obtained
in this section, whose main result is the following:

\begin{theorem}\label{teo:longtime} \mbox{}
 With $\T$ and $\widetilde{c}_1(\T)$ as before, the following hold true: 
 \begin{equation}
  \left(\frac{n\tau}{\alpha}\right)^{(n-1)/n}\left(\widetilde{c}_1(\T)\right)^{n-1} = 1 + 
  (n-1)\left(1-\textstyle{\frac{1}{n}}\right)\frac{\log\T}{\T} + o\Bigl(\textstyle{\frac{\log\T}{\T}}\Bigr), 
  \ \mbox{\rm as $\T\to\infty$}. \label{longtime}
 \end{equation}
\end{theorem}

To prove this theorem we need to express $t$ in terms of $\T$ in (\ref{ctlimit}), and  we shall do this in the next few lemmas. 
The approach is analogous to the one used in the previous section in order to relate $t$ and $\zeta.$

\begin{lemma}\label{lem:taut} \mbox{}
For $\T(t)$ defined by (\ref{eq:tau}), and being $c_1(t)$ the first component 
of any solution $(c_1,c_n,c_{n+1},\ldots)$ of (\ref{system}), the following holds
\begin{equation}
 \T(t) = \textstyle{\frac{n+1}{n}}\left(\frac{\alpha}{n+1}\right)^{\frac{1}{n+1}}t^{\frac{n}{n+1}} + \frac{n(n-1)}{n+1}\log t + 
 o(\log t), \quad\mbox{ as $t\to\infty$.} \label{eq:tauoft}
\end{equation}
\end{lemma}

\begin{proof}
Let us  consider the time scale (\ref{eq:tau}) and write it as
\begin{equation}
 \T(t) = \int_0^tc_1(s)ds = \T(T) + \int_T^tc_1(s)ds.  \label{eq:tauT}
\end{equation}
Remembering that $x(t):=c_1(t)$, we deduce from (\ref{ctlimit}) that, for all $\ve>0,$ there exists $T=T(\ve)$ such that,
for all $t>T$,
\begin{equation}
 \frac{n+1}{n(n-1)}\,t\,\left(c_1(t)-\left(\frac{\alpha}{n+1}\right)^{\frac{1}{n+1}}t^{-\frac{1}{n+1}}\right)
 \in [1-\ve, 1+\ve]. \label{eq:epsilonbounds}
\end{equation}
Let us consider the upper bound case in (\ref{eq:epsilonbounds}) (the lower bound case is analogous). Substituting it in (\ref{eq:tauT})
one gets
\begin{eqnarray*}
 \T(t)-\T(T) & = & \int_T^tc_1(s)ds \\
 & \leq & \left(\frac{\alpha}{n+1}\right)^{\frac{1}{n+1}}\int_T^ts^{-\frac{1}{n+1}}ds + (1+\ve)\frac{n(n-1)}{n+1}\int_T^t\frac{1}{s}ds\\
 & = & \textstyle{\frac{n+1}{n}}\left(\frac{\alpha}{n+1}\right)^{\frac{1}{n+1}}\left(t^{\frac{n}{n+1}}-T^{\frac{n}{n+1}}\right) + 
 (1+\ve)\frac{n(n-1)}{n+1}\log\frac{t}{T},
\end{eqnarray*}
from which follows
\[
 \T(t) \leq\textstyle{\frac{n+1}{n}}\left(\frac{\alpha}{n+1}\right)^{\frac{1}{n+1}}t^{\frac{n}{n+1}} + (1+\ve)\frac{n(n-1)}{n+1}\log t + 
 \mathcal{O}(1), \quad\mbox{\rm as $t\to\infty$.}
\]
By considering the lower bound in (\ref{eq:epsilonbounds}) we obtain the reverse inequality, with $1+\ve$ substituted by $1-\ve$, 
and the result (\ref{eq:tauoft}) follows by the arbitrariness of $\ve$.
 \end{proof}

\begin{lemma}\label{lem:ttau} 
With the assumptions of Lemma~\ref{lem:taut} we have, as $\T\to\infty$,
\begin{eqnarray}
 t(\T) & = & (\textstyle{\frac{n}{n+1}})^{\frac{n+1}{n}}\left(\textstyle{\frac{n+1}{\alpha}}\right)^{\frac{1}{n}}\T^{\frac{n+1}{n}} -
 (n-1)(\textstyle{\frac{n}{n+1}})^{\frac{n+1}{n}}\left(\textstyle{\frac{n+1}{\alpha}}\right)^{\frac{1}{n}}\T^{\frac{1}{n}}\log \T + 
  o(\T^{\frac{1}{n}}\log\T). \nn \\
  & & \mbox{} \label{eq:toftau}
\end{eqnarray}
\end{lemma}

\begin{proof} Let $B:=\frac{n+1}{n}\left(\frac{\alpha}{n+1}\right)^{\frac{1}{n+1}}$ and $D:=\frac{n(n-1)}{n+1},$ and let us write
(\ref{eq:tauoft}) as
\begin{equation}
 \T(t)=Bt^{\frac{n}{n+1}} + D(\log t)(1+o(1)), \quad\mbox{\rm as $t, \T\to\infty$.} \label{tauoftAB}
\end{equation}

Thus $\T(t) = Bt^{\frac{n}{n+1}} (1+ o(1)),$ and so $t= B^{-\frac{n+1}{n}}\T^{\frac{n+1}{n}}(1+o(1)).$
Substituting this expression in the logarithm term of (\ref{tauoftAB}) we get, as $t,\T\to\infty,$
\begin{eqnarray*}
 \T(t) & = & Bt^{\frac{n}{n+1}} + D\log \left(B^{-\frac{n+1}{n}}\T^{\frac{n+1}{n}}(1+o(1))\right) (1+ o(1))\\
 & = & Bt^{\frac{n}{n+1}} + D\log \left(B^{-\frac{n+1}{n}}\T^{\frac{n+1}{n}}\right) + o(\log\T),
\end{eqnarray*}
and hence
\begin{eqnarray}
 t^{-\frac{n}{n+1}} & = & B\left(\T - \textstyle{\frac{n+1}{n}}D\log\T  + o(\log\T)\right)^{-1}\nonumber \\
 & = & B\T^{-1}\left( 1 - \left(\textstyle{\frac{n+1}{n}}D\T^{-1}\log\T + o(\textstyle{\frac{\log\T}{\T}})\right)\right)^{-1}\nonumber \\
 & = & B\T^{-1}\left(1+\textstyle{\frac{n+1}{n}}D\T^{-1}\log\T + o(\textstyle{\frac{\log\T}{\T}})\right).  \label{tpower<0}
\end{eqnarray}
Raising (\ref{tpower<0}) to the power $-\textstyle{\frac{n+1}{n}}$ and using Newton's binomial series in the right-hand side we
arrive at

\begin{equation}
 t = B^{-\frac{n+1}{n}}\T^{\frac{n+1}{n}} - \left(\textstyle{\frac{n+1}{n}}\right)^2DB^{-\frac{n+1}{n}}\T^{\frac{1}{n}}\log\T + 
 o(\T^{\frac{1}{n}}\log\T),\label{tversustau}
\end{equation}
and substituting $B$ and $D$ back into (\ref{tversustau}) results in (\ref{eq:toftau}).
\end{proof}

We are now ready to prove Theorem~\ref{teo:longtime}.

\begin{proof}(of Theorem~\ref{teo:longtime}).
From (\ref{eq:ctilde}) and (\ref{ctlimit}) we have, as $t, \T\to +\infty$,

\begin{eqnarray*}
\left(\textstyle{\frac{n\tau}{\alpha}}\right)^{\frac{n-1}{n}}\!\!\!\left(\widetilde{c}_1(\T)\right)^{n-1} \!\!
& = &\!\! \left(\textstyle{\frac{n\tau}{\alpha}}\right)^{\frac{n-1}{n}}\!\!\!\left(
\left(\textstyle{\frac{\alpha}{n+1}}\right)^{\frac{1}{n+1}}t^{-\frac{1}{n+1}} + \textstyle{\frac{n(n-1)}{n+1}}t^{-1}
 + o(t^{-1})\right)^{n-1}\!\!\!\\
 & = & \left(\textstyle{\frac{n\tau}{\alpha}}\right)^{\frac{n-1}{n}}\left(\textstyle{\frac{\alpha}{n+1}}\right)^{\frac{n-1}{n+1}}t^{-\frac{n-1}{n+1}}\times\\
 &   & \qquad\quad\times \left(1 + \textstyle{\frac{n(n-1)}{n+1}}\left(\textstyle{\frac{n+1}{\alpha}}\right)^{\frac{1}{n+1}}t^{-\frac{n}{n+1}}
 + o\left(t^{-\frac{n}{n+1}}\right)\right)^{n-1}\\
 & = & \left(\textstyle{\frac{n}{\alpha}}\right)^{\frac{n-1}{n}}\left(\textstyle{\frac{\alpha}{n+1}}\right)^{\frac{n-1}{n+1}}
 \T^{\frac{n-1}{n}}t^{-\frac{n-1}{n+1}}\times\\
 &   & \qquad\quad\times \left(1 + \textstyle{\frac{n(n-1)^2}{n+1}}\left(\textstyle{\frac{n+1}{\alpha}}\right)^{\frac{1}{n+1}}t^{-\frac{n}{n+1}}
 + o\left(t^{-\frac{n}{n+1}}\right)\right),
\end{eqnarray*} 
where we used Newton's binomial expansion in the last equality.
We can now apply Lemma~\ref{lem:ttau} to write the right-hand side in terms of $\T$, and after a few computations analogous to
those described above, (\ref{longtime}) arises.
\end{proof}


\section{Rate of convergence to the similarity profile}\label{sec:crsp}

In this section we prove the  paper's main result, Theorem~\ref{teo:main}.

In \cite{costin} it was proved that the limit of (\ref{convratelhs}) 
is equal to zero, thus generalizing a similar result first proved in \cite{crw} in the special case $n=2.$ 
For the proof of Theorem~\ref{teo:main} we shall apply the approach used in \cite{crw,costin}, consisting in
exploring the representation formula (\ref{cjtausolution}). In those papers the information that was needed 
was just the first term in Stirling's expansion of the gamma function (to estimate the factorials)
and the first term in the long time behaviour of the solution $\widetilde{c}_1$ which, in \cite{crw}
was obtained by center manifold methods, and in \cite{costin} by asymptotic analysis methods. Here we
shall essentially make use of the results of Sections \ref{sec:cma} and \ref{sec:cts}, in particular
(\ref{longtime}) proved in Theorem~\ref{teo:longtime}, in order 
to get the needed higher order information.

Before starting the proof of Theorem~\ref{teo:main} we briefly point out some of its consequences.

\begin{corollary}
 From Theorem~\ref{teo:main} we immediately conclude the following:
 \begin{description}
  \item[i.] The convergence is \emph{not} uniform in $\eta\in (0,\infty)\setminus\{1\}$.
  \item[ii.] For $\eta\in\Xi,$ a compact set not containing $1$ and such that $\Xi\cap (0,1)\neq \emptyset$ and 
  $\Xi\cap (1,\infty)\neq \emptyset,$ if $1<\mu < 2-\frac{1}{n},$ the solution converges to the similarity profile 
  $\Phi_1(\eta)$ at a power rate $\tau^{\frac{n-1}{n}-\mu}$.
  \item[iii.] Under the same conditions, if $\mu \geq2-\frac{1}{n}$ the solution converges 
  to the similarity profile $\Phi_1(\eta)$ at a rate $\frac{\log\tau}{\tau}$.
 \end{description}
\end{corollary}

\begin{remark}
Observe that Theorem~\ref{teo:main} implies that for initial conditions $(c_j(0))$  decaying fast enough with 
$j$, namely, for $\mu \geq2-\frac{1}{n}$,
the information about the initial data is lost in the similarity limit, in the sense that neither the bounds 
$\rho_1$ and $\rho_2$, nor the decay rate $\mu,$ and 
obviously no other details of the initial condition, are reflected in the limit or in the rate at which the limit is approached. 

However, if the initial data is slowly decaying, namely if $1<\mu < 2-\frac{1}{n},$
then, although information about the initial condition is also lost in the similarity limit, the rate at which 
this limit is approached still retains 
information about the initial condition.

Furthermore, if one can measure the decay at a single value of the similarity variable $\eta>1$ 
(and not only the overall rate in a set $\Xi$ of values of $\eta$) then the rate of decay of the 
initial condition for large clusters is always reflected in the rate of
convergence to the similarity profile $\Phi_1(\eta).$

We think this is an interesting behaviour that, as far as we know, seems to be the first time it is observed in 
coagulation type systems.
\end{remark}

\begin{proof} (of Theorem~\ref{teo:main}).
 As in \cite{crw,costin}, we start by considering the case of monomeric initial conditions,
 in which case the sum in the 
 right hand side of (\ref{cjtausolution}) is zero for all time $\tau$.
 
 Consider the function $\varphi_1: [n, \infty)\times[0,\infty)$ defined by
 \begin{equation}
  \varphi_1(x,\tau) := \frac{\left(\frac{n\tau}{\alpha}\right)^{(n-1)/n}}{\Gamma(x-n+1)}
 \int_0^{\tau}\left(\widetilde{c}_1(\T-s)\right)^{n-1}s^{x-n}e^{-s}ds.\label{fi1} 
 \end{equation}
Comparing this expression with the integral in the right hand side of (\ref{cjtausolution}) we see that, 
when $x=j\in\Nb\cap[n, \infty),$ 
$\varphi_1(j,\tau) = \left(\frac{n\tau}{\alpha}\right)^{(n-1)/n}\widetilde{c}_j(\T).$ 
So, we are going to study the rate of convergence
of $\varphi_1(\eta\T,\T)$ to $\Phi_1(\eta)$ as $\T\to\infty$, for fixed $\eta\in(0,\infty)\setminus\{1\}.$

Changing the integration variable in (\ref{fi1}) $s\mapsto y=s/\T$, and
using  (\ref{longtime}) to write, for $0<y<1,$
\[
 \left(\frac{n\tau(1-y)}{\alpha}\right)^{(n-1)/n}\!\!\!\left(\widetilde{c}_1(\T(1-y))\right)^{n-1} = 1 + f_n(\T(1-y)), 
\]
where $f_n(\T(1-y)) = (n-1)\left(1-\textstyle{\frac{1}{n}}\right)\,\frac{\log\T(1-y)}{\T(1-y)} + 
o\left(\textstyle{\frac{\log\T}{\T}}\right)$  as $\T\to\infty,$
the expression (\ref{fi1}) can be written as 
\begin{equation}
  \varphi_1(\eta\T,\tau) = \frac{\T^{1-n+\eta\T}}{\Gamma(\eta\T-n+1)}
 \int_0^{1}(1 + f_n(\T(1-y)))\frac{e^{\T(\eta\log y-y)}}{y^n(1-y)^{(n-1)/n}}dy. \label{fi2} 
 \end{equation}
Now, in order to proceed, we need the following estimates, obtained from Stirling's asymptotic formula, 
and valid as $\T\to\infty,$
\begin{eqnarray}
 \Gamma(\eta\T-n+1) &=& \underbrace{\frac{\Gamma(\eta\T)}{(\eta\T-1)(\eta\T-2)\cdots(\eta\T-(n-1))}}_{(n-1)\, \mbox{\rm terms}} 
 \nonumber \\
 & = & \frac{e^{-\eta\T}(\eta\T)^{\eta\T-\frac{1}{2}}\sqrt{2\pi}\left(1+ \frac{1}{12\eta\T} + {\mathcal O}(\T^{-2})\right)}
 {(\eta\T)^{n-1}\left(1-\frac{n(n-1)}{2}\frac{1}{\eta\T}+ {\mathcal O}(\T^{-2})\right)} \nonumber \\
 & = & \sqrt{2\pi}e^{-\eta\T}(\eta\T)^{\eta\T-n+\frac{1}{2}}
 \frac{1+ \frac{1}{12\eta\T} + {\mathcal O}(\T^{-2})}{1-\frac{n(n-1)}{2}\frac{1}{\eta\T}+ {\mathcal O}(\T^{-2})}. \label{estimate1}
\end{eqnarray}
Using the formula for the sum of a geometric series we can write
\begin{eqnarray}
\lefteqn{\left(1-\frac{n(n-1)}{2}\frac{1}{\eta\T}+ {\mathcal O}(\T^{-2})\right)\left(1+ \frac{1}{12\eta\T} + {\mathcal O}(\T^{-2})\right)^{-1} =}
\nonumber\\ & = &
\left(1-\frac{n(n-1)}{2}\frac{1}{\eta\T}+ {\mathcal O}(\T^{-2})\right)\left(1- \frac{1}{12\eta\T} - {\mathcal O}(\T^{-2})\right) \nonumber \\
& = & 1- \left(\frac{1}{12}+\frac{n(n-1)}{2}\right)\frac{1}{\eta\T} + {\mathcal O}(\T^{-2}). \label{estimate2}
\end{eqnarray}
Substituting (\ref{estimate1}) into (\ref{fi2}) and using (\ref{estimate2}), equation (\ref{fi2}) becomes
\begin{eqnarray}
 \varphi_1(\eta\T,\tau) & = &\frac{1}{\sqrt{2\pi}}\eta^{n-\frac{1}{2}-\eta\T}\T^{\frac{1}{2}}
 \left(1-\frac{1+6n(n-1)}{12}\frac{1}{\eta\T}+ {\mathcal O}(\T^{-2})\right)\times  \nonumber \\
 & & \times\int_0^{1}(1 + f_n(\T(1-y)))\frac{e^{\T(\eta\log y-y+\eta)}}{y^n(1-y)^{(n-1)/n}}dy \nonumber \\
 & = & J_1(\eta,\T) - J_2(\eta,\T) + J_3(\eta,\T) - J_4(\eta,\T), \label{fi3}
\end{eqnarray} 
where $J_k(\eta,\T),\,k =1, \ldots, 4,$ are defined by
\begin{eqnarray}
 & & J_1(\eta,\T)
  :=  \frac{\eta^{n-\eta\T}}{\sqrt{2\pi\eta}}\T^{\frac{1}{2}}
 \int_0^{1}\frac{e^{\T(\eta\log y-y+\eta)}}{y^n(1-y)^{(n-1)/n}}dy,  \label{Int1} \\
 & & J_2(\eta,\T)  :=   \frac{\eta^{n-\eta\T}}{\sqrt{2\pi\eta}}\T^{\frac{1}{2}}
 \left(\frac{1+6n(n-1)}{12\eta\T}+ {\mathcal O}(\T^{-2})\right)\times \label{Int2} \\
 &  & \qquad\qquad\qquad\times 
 \int_0^{1}\frac{e^{\T(\eta\log y-y+\eta)}}{y^n(1-y)^{(n-1)/n}}dy, \nonumber  \\
& & J_3(\eta,\T)  :=   \frac{\eta^{n-\eta\T}}{\sqrt{2\pi\eta}}\T^{\frac{1}{2}}
 \int_0^{1}f_n(\T(1-y))\frac{e^{\T(\eta\log y-y+\eta)}}{y^n(1-y)^{(n-1)/n}}dy , \label{Int3}\\
& & J_4(\eta,\T)  :=   \frac{\eta^{n-\eta\T}}{\sqrt{2\pi\eta}}\T^{\frac{1}{2}}
 \left(\frac{1+6n(n-1)}{12\eta\T}+{\mathcal O}(\T^{-2})\right)\times \label{Int4}\\
 &  & \qquad\qquad\qquad\times \int_0^{1}f_n(\T(1-y))\frac{e^{\T(\eta\log y-y+\eta)}}{y^n(1-y)^{(n-1)/n}}dy. \nonumber  
\end{eqnarray}
We now study the limit as $\T\to\infty$ of each of these functions separately.

\begin{lemma}\label{lem:Int1} 
For each fixed $\eta\in (0,\infty)\setminus\{1\},$ the following holds:
 $$J_1(\eta,\T) = \left(\Phi_1(\eta) + {\mathcal O}(1)\T^{-1}\right)\mathds{1}_{(0,1)}(\eta) + \alpha_1(\eta,\T),$$  
 as $\T\to\infty$, and $\alpha_1(\eta,\T)\to 0$ beyond all orders.
\end{lemma}

\begin{proof}
 The proof of this result follows \cite[Section~5.1]{crw} and \cite[Propositions~6 and~7]{costin}, and
is exactly the same as those proofs with the function $\psi$ in those papers identically equal to $1$.
We need only to take the additional care of keeping track of the higher order terms, in order to be sure
they are indeed negligible relative to the dominant terms arising from the other $J_k$ in (\ref{fi3}).

If $\eta>1$ and $\T>\frac{n}{\eta-1}$ the function $y\mapsto(\eta\T-n)\log y -y\tau$ is strictly increasing  
in $y\in (0,1).$ 
Thus, for these values of $\T$, $y^{-n}e^{\T(\eta\log y-y)} = e^{(\eta\T-n)\log y -y\T} \leq e^{-\T},$
and we can estimate (\ref{Int1}) as follows:
\begin{eqnarray*}
 J_1(\eta,\T)& = & \frac{\eta^{n-\eta\T}}{\sqrt{2\pi\eta}}\T^{\frac{1}{2}}
 \int_0^{1}\frac{e^{\T(\eta\log y-y+\eta)}}{y^n(1-y)^{(n-1)/n}}dy \\
 & \leq& \frac{1}{\sqrt{2\pi\eta}}\eta^{n}\T^{\frac{1}{2}}e^{-\T(\eta\log\eta-\eta+1)}\int_0^{1}\frac{1}{(1-y)^{(n-1)/n}}dy\\
 & = & \frac{n}{\sqrt{2\pi\eta}}\eta^{n}\T^{\frac{1}{2}}e^{-\T(\eta\log\eta-\eta+1)}.
\end{eqnarray*}

If $\eta\in (0,1)$ the approach is slightly more involved, but is also essentially the one used in \cite{crw,costin}.
Fixing an $\ve \in \left(0, \min\left\{\eta e^{-1}, 1-\eta\right\}\right),$ we decompose the integral in (\ref{Int1}) as follows,
\begin{equation}
 J_1(\eta,\T) :=  \frac{\eta^{n-\eta\T}}{\sqrt{2\pi\eta}}\T^{\frac{1}{2}}
 \left(\int_0^{\ve}+\int_{\ve}^{1-\ve}+\int_{1-\ve}^1\right)\frac{e^{\T(\eta\log y-y+\eta)}}{y^n(1-y)^{(n-1)/n}}dy.\label{Int1-1}
\end{equation}

For the contributions arising from the first and third integrals in (\ref{Int1-1}), the computations in \cite[p. 385]{crw}
show that both can be bounded above by a function which exponentially decreases to zero as $\T\to\infty$, namely:
\begin{eqnarray}
 \lefteqn{\frac{\eta^{n-\eta\T}}{\sqrt{2\pi\eta}}\T^{\frac{1}{2}}
 \left(\int_0^{\ve}+\int_{1-\ve}^1\right)\frac{e^{\T(\eta\log y-y+\eta)}}{y^n(1-y)^{(n-1)/n}}dy =}\nonumber \\
 & \leq& \left(\frac{n\eta^{n}}{\sqrt{2\pi\eta}}(1-\eta^{1/n})e^{n(1-\log\eta)}\right)\T^{\frac{1}{2}}e^{-\T\eta/e}\, + \\
 & &  + \left(\frac{n\eta^{n}}{\sqrt{2\pi\eta}}\frac{\ve^{1/n}}{(1-\ve)^n}\right)\T^{\frac{1}{2}}e^{-\T g_3(1-\ve)},
\end{eqnarray}
for all sufficiently large $\T$,
where\footnote{This notation was employed in \cite{crw,costin} and is maintained here for convenience of the reader.} 
$g_3(y) := (\eta\log\eta-\eta)-(\eta\log y-y).$ 
 
For the remaining integral in (\ref{Int1-1}) we again follow \cite{crw}, as modified by \cite{costin}. Without any change
to what was done in those papers, an application of Laplace's method \cite[Chap. 3]{Miller} results in the following, 
valid when $\T\to\infty,$
\begin{equation}
 \frac{\eta^{n-\eta\T}}{\sqrt{2\pi\eta}}\T^{\frac{1}{2}}
 \int_{\ve}^{1-\ve}\frac{e^{\T(\eta\log y-y+\eta)}}{y^n(1-y)^{(n-1)/n}}dy = \frac{1}{(1-\eta)^{(n-1)/n}} + {\mathcal O}(1)\T^{-1}.
 \label{eq:useLaplace1}
\end{equation}
This completes the proof of the lemma.
\end{proof}

\begin{lemma}\label{lem:Int2} 
For each fixed $\eta\in (0,\infty)\setminus\{1\},$ the following holds:
 $$J_2(\eta,\T) = {\mathcal O}(1)\T^{-1}\mathds{1}_{(0,1)}(\eta) + \alpha_2(\eta,\T),$$  
 as $\T\to\infty$, and $\alpha_2(\eta,\T)\to 0$ beyond all orders.
\end{lemma}

\begin{proof}
 The proof is exactly equal to that of Lemma~\ref{lem:Int1}, with the exception that now all the terms will be 
 multiplied by $\frac{1+6n(n-1)}{12\eta\T}
 + {\mathcal O}(\T^{-2}).$
 Hence the result immediately follows.
\end{proof}

\begin{lemma}\label{lem:Int3} 
For each fixed $\eta\in (0,\infty)\setminus\{1\},$ the following holds:
 $$J_3(\eta,\T) = (n-1)\left(1-\textstyle{\frac{1}{n}}\right)\frac{1}{(1-\eta)^{(n-1)/n}}\frac{\log((1-\eta)\tau)}{(1-\eta)\tau}
 (1+o(1))\mathds{1}_{(0,1)}(\eta) + \alpha_3(\eta,\T),$$  
as $\T\to\infty,$ and $\alpha_3(\eta,\T)\to 0$ beyond all orders.
\end{lemma}

\begin{proof}
 The proof of this lemma follows the steps taken in the proof of Lemma~\ref{lem:Int1} with the function inside the
 integral now containing the additional multiplicative factor 
 $$f_n((1-y)\T) = (n-1)\left(1-\textstyle{\frac{1}{n}}\right)\,\frac{\log((1-y)\T)}{(1-y)\T} + 
o\left(\textstyle{\frac{\log\T}{\T}}\right),\quad \mbox{\rm  as $\T\to\infty$.}$$ 

As before, the cases $\eta>1$ and $\eta\in (0,1)$ need to be treated separately, and the last one requires the
splitting of the integral into three terms, two of them dealing with regions near the boundary. 
Since $\frac{\log((1-y)\T)}{(1-y)\T}\to 0$ as $\T\to\infty$,
the asymptotically exponentially small bounds when $\T\to\infty$ obtained in the proof of 
Lemma~\ref{lem:Int1} for the cases $\eta>1$, and for the ``boundary
integrals'' $\int_0^{\ve}$ and $\int_{1-\ve}^1$ in the case $\eta\in (0,1)$, still hold in the present case.

Thus, we are left with estimating the integral $\int_{\ve}^{1-\ve}$, which can again be done by Laplace's method 
exactly as before, giving the same result as (\ref{eq:useLaplace1}) with the additional multiplicative factor 
$\frac{\log((1-\eta)\T)}{(1-\eta)\T}.$
This completes the proof.
\end{proof}

\begin{lemma}\label{lem:Int4} 
For each fixed $\eta\in (0,\infty)\setminus\{1\},$ the following holds:
 $$J_4(\eta,\T) = {\mathcal O}(1)\frac{\log((1-\eta)\tau)}{(1-\eta)\tau}\,\frac{1}{\T}\mathds{1}_{(0,1)}(\eta) + 
 \alpha_4(\eta,\T),$$  
 as $\T\to\infty$, and $\alpha_4(\eta,\T)\to 0$ beyond all orders.
\end{lemma}

\begin{proof}
The proof of this lemma is related to that of Lemma~\ref{lem:Int3}, as the proof of Lemma~\ref{lem:Int2} 
is to that of Lemma~\ref{lem:Int1}: 
the expressions of the corresponding $J_k$ differ by the factor 
$\frac{1+6n(n-1)}{12\eta\T} + {\mathcal O}(\T^{-2}),$ which entails
the result.
\end{proof}

Putting together the results in Lemmas~\ref{lem:Int1}--\ref{lem:Int4} we conclude the proof of Theorem~\ref{teo:main} 
in the case of monomeric initial data,
that is: when $c_j(0)=0$ for all $j\geq n,$ and the only initial component potentially nonzero is $c_1(0).$

Let us now consider the case when the initial data is nonzero. More specifically we shall consider initial conditions satisfying the
following assumption:
\begin{description}
 \item[(H1)] $\exists \rho_1, \rho_2>0,\; \mu > 1:\; \forall j\geq n,\; j^\mu c_j(0)\in [\rho_1, \rho_2].$
\end{description}

So, let us now consider the sum in the right hand side of (\ref{cjtausolution}). We are interested in studying the behaviour of
\[
 \left(\frac{n\T}{\alpha}\right)^\frac{n-1}{n} e^{-\T}\sum_{k=n}^{j}\frac{\T^{j-k}}{(j-k)!}c_k(0),
\]
when $j,\,\tau \rightarrow +\infty,$ with $\eta=j/\tau\neq 1$ fixed.
As in \cite{crw,costin}, we shall write $\nu:=\eta^{-1}$ and use $\T=j\nu$ in the above expression.
We are then left with studying the limit, as $j\to\infty$ of
\begin{equation}
 \left(\frac{n\nu}{\alpha}\right)^\frac{n-1}{n} j^{\frac{n-1}{n}}e^{-j\nu}
 \sum_{k=n}^{j}\frac{(j\nu)^{j-k}}{(j-k)!}c_k(0). \label{initialcontribution}
\end{equation}

Consider the cases $\nu>1$ and $\nu\in (0,1)$ separately.

When $\nu>1$ the limit is easy to handle, giving an asymptotically exponentially small contribution. In fact,
the situation is entirely analogous to the one in \cite[page 387]{crw}: using the upper bound in {\bf (H1)} and changing
the summation variable $k\mapsto\ell:=j-k$ in (\ref{initialcontribution}) we get
\begin{equation}
 \left(\frac{n\nu}{\alpha}\right)^\frac{n-1}{n} j^{\frac{n-1}{n}}e^{-j\nu}
 \sum_{k=n}^{j}\frac{(j\nu)^{j-k}}{(j-k)!}c_k(0) 
\leq\frac{\rho_2}{n^{\mu}}\left(\frac{n\nu}{\alpha}\right)^\frac{n-1}{n} j^{\frac{n-1}{n}}e^{-j\nu}
 \sum_{\ell=0}^{j-n}\frac{(j\nu)^{\ell}}{\ell!},
\label{initialcontributionbound1}
\end{equation}
and from here the computations in \cite{crw} apply almost verbatim since (\ref{initialcontributionbound1}) with
$n=2$ reduces to the expression in \cite{crw}. The final result is the bound
\begin{equation}
 \left(\frac{n\nu}{\alpha}\right)^\frac{n-1}{n} \!\!\!j^{\frac{n-1}{n}}e^{-j\nu}
 \sum_{k=n}^{j}\frac{(j\nu)^{j-k}}{(j-k)!}c_k(0) 
< \frac{1}{\nu^n\sqrt{2\pi}}\left(\frac{n\nu}{\alpha}\right)^\frac{n-1}{n} \!\!\!j^{\frac{3}{2}-\frac{1}{n}}
e^{-j(\nu-1-\log\nu)}(1+o(1)),\label{initialcontributionbound2}
\end{equation}
as $j\to\infty.$ Since $\nu>1,$ we have $\nu-1-\log\nu>0,$ and thus the upper bound is asymptotically
exponentially small in the large $j$ limit.

When $\nu\in(0,1)$ we again follow \cite[page 388]{crw}: again changing the summation
variable $k\mapsto\ell:=j-k$ and using the upper bound in {\bf (H1)}, we can choose 
a fixed $\beta\in\left(\nu e^{1-\nu}, \min\{\nu e,1\}\right)$ and split the sum into 
a ``small'' $\ell$ and a ``large'' $\ell$ contribution,
\[
 \sum_{\ell=0}^{j-n}\frac{(j\nu)^{\ell}}{\ell!(j-\ell)^{\mu}} =  
 \sum_{0\leq\ell \leq\beta j}\frac{(j\nu)^{\ell}}{\ell!(j-\ell)^{\mu}}
 + 
 \sum_{\beta j< \ell \leq j-n}\frac{(j\nu)^{\ell}}{\ell!(j-\ell)^{\mu}}.
\]
The analysis of each of these terms proceeds exactly as in \cite[page 388]{crw}, and concluding that the contribution from the 
``large'' $\ell$ is asymptotically exponentially small when $j\to\infty$, and the one due to the ``small'' $\ell$ sum
being bounded above by
\begin{equation}
 \frac{\rho_2}{\left((1-\beta)n\right)^{\mu}}\left(\frac{n}{\alpha}\right)^{\frac{n-1}{n}}
 \eta^{-\mu}\T^{\frac{n-1}{n}-\mu},\label{upper}
\end{equation}
where we have used the fact that $j\nu=\T.$

To complete the proof we need to obtain a lower bound for the limit of (\ref{initialcontribution}) when $j\to\infty.$
Using the lower bound in {\bf (H1)} and the same change of variable $k\mapsto\ell$ as above, it is easy to obtain the bound
\begin{equation}
 \left(\frac{n\nu}{\alpha}\right)^\frac{n-1}{n}\!\!\!j^{\frac{n-1}{n}}e^{-j\nu}
 \sum_{k=n}^{j}\frac{(j\nu)^{j-k}}{(j-k)!}c_k(0) 
> \rho_1\left(\frac{n\nu}{\alpha}\right)^{\frac{n-1}{n}}\!\!\!j^{\frac{n-1}{n}-\mu}
e^{-j\nu}\sum_{\ell=0}^{j-n}\frac{(j\nu)^{\ell}}{\ell!}.\label{initialcontributionbound3}
\end{equation}
Now write
\[
 e^{-j\nu}\sum_{\ell=0}^{j-n}\frac{(j\nu)^{\ell}}{\ell!} =
 e^{-j\nu}\sum_{\ell=0}^{j}\frac{(j\nu)^{\ell}}{\ell!} - e^{-j\nu}\!\!\sum_{\ell=j-n+1}^{j}\!\!\frac{(j\nu)^{\ell}}{\ell!}.
\]
Since $\nu\in (0,1)$, the first term in the right hand side converges to $1$ as $j\to\infty$ (cf., e.g., \cite[formula 6.5.34]{AS}).
For the second term, the sum of which has $n$ terms, a finite and fixed number independent of $j$, we have
\begin{eqnarray*}
 e^{-j\nu}\!\!\!\sum_{\ell=j-n+1}^{j}\!\!\frac{(j\nu)^{\ell}}{\ell!} & = & 
 e^{-j\nu}\left(\textstyle{\frac{(j\nu)^j}{j!}} + \textstyle{\frac{(j\nu)^{j-1}}{(j-1)!}} + \textstyle{\frac{(j\nu)^{j-2}}{(j-2)!}}+ 
 \cdots +  \textstyle{\frac{(j\nu)^{j-n+1}}{(j-n+1)!}}\right)\\
 & = & e^{-j\nu}\frac{(j\nu)^j}{j!}\left(1 + \textstyle{\frac{1}{\nu}} + \textstyle{\frac{1}{\nu^2}}\left(1-\textstyle{\frac{1}{j}}\right) + \cdots 
 + \textstyle{\frac{1}{\nu^{n-1}}}\left(1+ {\mathcal O}(j^{-1})\right) \right)\\
 & = & e^{-j\nu}\frac{(j\nu)^j}{j!}\left(1 + \textstyle{\frac{1}{\nu}} + \textstyle{\frac{1}{\nu^2}} + \cdots 
 + \textstyle{\frac{1}{\nu^{n-1}}} \right)\left(1+ {\mathcal O}(j^{-1})\right)\\
 & = & e^{-j\nu}\frac{(j\nu)^j}{j!}\textstyle{\frac{1-\nu^{n}}{(1-\nu)\nu^{n-1}}}\left(1+ {\mathcal O}(j^{-1})\right).
\end{eqnarray*}
Using Stirling's expansion we have $e^{-j\nu}\frac{(j\nu)^j}{j!} = \frac{1}{\sqrt{2\pi j}}e^{j(1-\nu+\log\nu)}(1+{\mathcal O}(j^{-1}))$.
From $\nu\in(0,1)$ it follows that $1-\nu+\log \nu<0$, and thus
\(
 e^{-j\nu}\!\displaystyle{\sum_{\ell=j-n+1}^{j}}\!\!\frac{(j\nu)^{\ell}}{\ell!}
\)
converges to zero exponentially fast as $j\to\infty.$ This implies that 
(\ref{initialcontributionbound3}) can be written as
\begin{equation}
 \left(\frac{n\nu}{\alpha}\right)^\frac{n-1}{n} \!\!\!j^{\frac{n-1}{n}}e^{-j\nu}
 \sum_{k=n}^{j}\frac{(j\nu)^{j-k}}{(j-k)!}c_k(0) 
> \rho_1\left(\frac{n\nu}{\alpha}\right)^\frac{n-1}{n} \!\!\!j^{\frac{n-1}{n}-\mu}
(1+o(1)).\label{initialcontributionbound4}
\end{equation}
Hence, a lower bound for the limit of (\ref{initialcontribution}) as $j\to\infty$ is
\begin{equation}
 \rho_1\left(\frac{n}{\alpha}\right)^{\frac{n-1}{n}}\!\!\!\eta^{-\mu}\T^{\frac{n-1}{n}-\mu}(1+o(1)),\label{lower}
\end{equation}
where we have used the fact that $j\nu=\T.$
From (\ref{initialcontributionbound3}) and (\ref{initialcontributionbound4}) the term (\ref{convrate2}) follows. 

This concludes the proof of Theorem~\ref{teo:main}.
\end{proof}



\end{document}